\documentclass{article}
\usepackage{amsmath}
\usepackage{amsthm}
\usepackage{amsfonts}
\usepackage{amssymb}
\usepackage{graphicx}

\newtheorem*{remark}{Remark}
\newtheorem{corollary}{Corollary}

\begin{document}

\title{Power sums of the roots of gap polynomials irreducible numerical semigroups}
\author{Pavel Grinfeld \\
Drexel University}
\maketitle

\begin{abstract}
For a numerical semigroup $S$ with Frobenius number $F$ and gaps $G$, a polynomial $P\left(z\right)=\sum_{g\in G}c_{g}z^{g}$ of degree $F$ is called a gap polynomial. We show that for a symmetric numerical semigroup, the power sums of the roots of any derivative of a gap polynomial vanish for every power $g$ in $G$. For a pseudo-symmetric numerical semigroup, the same result holdsif the term $z^{F/2}$ is omitted from the polynomial.
\end{abstract}

\section{Introduction}

Numerical semigroups are an active and rapidly expanding area of research.
Beyond their fundamental structural properties, investigation of connections
between numerical semigroups and the roots of polynomials have yielded
important results \cite{Ciolan2016Cyclotomic}, \cite{Moree2014Numerical}, 
\cite{BorziHerreraPoyatosMoree2021}. Furthermore, the connection between
numerical semigroups and symmetric polynomials has been receiving a great
deal of attention \cite{Fel2015}, \cite{Fel2021}. In this paper, we build
upon these connections by investigating the power sums of the roots of gap polynomials, as well as their derivatives and integrals, for symmetric and pseudo-symmetric numerical semigroups.

The references \cite{RamirezAlfonsin2005}, \cite{Rosales2009}, \cite%
{Assi2016} offer excellent to the subject along with novel and significant
results. Here, we will limit ourselves to definitions that are relevant to
our central result.

A \textit{numerical semigroup} $S$ is a set of non-negative integers that is
closed under addition and contains $0$ along with all but finitely many
natural numbers. For example, the numerical semigroup 
\begin{equation}
S=\left\{ 0,4,5,8,9,10,12,13,14,\ldots \right\} ,
\end{equation}%
contains all integers except for the \textit{set of gaps} 
\begin{equation}
G=\left\{ 1,2,3,6,7,11\right\} .
\end{equation}%
The largest gap, which is $11$ for the present example, is the \textit{%
Frobenius number} $F$.

A numerical semigroup $S$ is said to be symmetric if it contains one and
only one of the numbers $i$ and $F-i$. The above semigroup $S$ is indeed
symmetric, which is easily verified directly, and is also evident from the
central symmetry in the following table showing its first few elements and
gaps: 
\begin{equation}
\begin{array}{cccccccccccc}
0 &  &  &  & 4 & 5 &  &  & 8 & 9 & 10 &  \\ 
& 1 & 2 & 3 &  &  & 6 & 7 &  &  &  & 11%
\end{array}%
\end{equation}

A numerical semigroup with an even Frobenius number $F$ cannot be symmetric
since $F/2$ would need to be simultaneously a gap and in $S$. Accordingly, a
numerical semigroup with an even Frobenius number $S$ is said to be \textit{%
pseudo-symmetric} if the condition of symmetry holds for all gaps except $%
F/2 $.

Pseudo-symmetric semigroups $S$ have the additional property if $i$ is in $S$
then so is $F/2+i$. Indeed, for $i>F/2$, this statement is obvious. For $i<S$%
, suppose that $F/2+i$ is not in $S$. Then, by symmetry, $F-\left(
F/2+i\right) =F/2-i$ is in $S$ and therefore, so is $i+\left( F/2-i\right)
=F/2$. However, $F/2$ must be gap and therefore we have arrived at a
contradiction.

For a numerical semigroup with an odd Frobenius number $F$, define a \textit{%
gap polynomial} $P\left( z\right) $ by 
\begin{equation}
P\left(z\right) =\sum_{g\in G}c_{g}z^{g}\qquad\text{ where }c_{F}\neq0.
\end{equation}
For a numerical semigroup with an even Frobenius number $F$, define a 
\textit{pseudo-gap polynomial} by additionally omitting the exceptional $%
z^{F/2}$ term, i.e. 
\begin{equation}
P\left( z\right) =\sum_{\substack{ g\in G  \\ g\neq F/2}}c_{g}z^{g}\qquad%
\text{ where }c_{F}\neq0.
\end{equation}

\section{Statement of the theorem}

Let $S$ be a symmetric (pseudo-symmetric) numerical semigroup with Frobenius
number $F$, and let $P\left( z\right) $ be a gap (pseudo-gap) polynomial of $%
S$ and let $P^{\left( m\right) }\left( z\right) $ denote its $m$-th
derivative. Suppose that $r_{1}$, $r_{2}$, $\ldots $, $r_{F-m}$ are the
roots of $P^{\left( m\right) }\left( z\right) $, counted with multiplicity,
and let $p_{k}$ denote their power sums, i.e. 
\begin{equation}
p_{k}=r_{1}^{k}+r_{2}^{k}+\cdots +r_{F-m}^{k}.
\end{equation}%
Then 
\begin{equation}
p_{g}=0  \label{pg}
\end{equation}%
for every gap $g$.

The theorem holds for all integers $m$ where, for a negative $m$, $P^{\left(
m\right) }\left( z\right) $ represents the $m$-th antiderivative of $P\left(
z\right) $ with arbitrary integration constants. Although the following
proof is given for nonnegative $m$, it can be trivially adjusted to negative 
$m$ as well.

\section{A proof of the theorem}

We will first assume that $S$ is symmetric and make the necessary adjustment
at the end for a pseudo-symmetric $S$. Let $P_{1}^{\left( m\right) }\left(
z\right) $ is the monic version of $P^{\left( m\right) }\left( z\right) $,
i.e. 
\begin{equation}
P_{1}^{\left( m\right) }\left( z\right) =\sum_{n=0}^{F-m}c_{n}z^{n},
\label{M}
\end{equation}%
where 
\begin{equation}
C_{n}=\left\{ 
\begin{array}{rl}
0, & \text{if }n<m\,\text{or }n\in G \\ 
\frac{n\left( n-1\right) \cdots \left( n-m+1\right) c_{n}}{F\left(
F-1\right) \cdots \left( F-m+1\right) c_{F}}, & \text{otherwise}%
\end{array}%
\right.
\end{equation}%
For $P_{1}^{\left( m\right) }\left( z\right) $, Newton's identities read 
\begin{equation}
p_{k}=-\sum_{i=1}^{k-1}p_{i}C_{F-k+i}-kC_{F-k}.  \label{NI}
\end{equation}%
Thanks to the fact that all subscripts in equation (\ref{NI}) are
independent of $m$, the proof below automatically applies to all orders of
the derivative.

We will now proceed by induction. Since the number $1$ is always a gap, we
must show that $p_{1}$ vanishes. Indeed, by symmetry, that fact that $1$ is
a gap implies that $F-1$ is in $S$, and therefore the corresponding
coefficient $C_{F-1}$ is $0$. Thus, $p_{1}$, given by 
\begin{equation}
p_{1}=-C_{F-1},
\end{equation}
vanishes as well.

Next, assume that all "gap" $p_{k}$ for $k$ below $g\in G$ vanish and
consider the equation for $p_{g}$: 
\begin{equation}
p_{g}=-\sum_{i=1}^{g-1}p_{i}C_{F-g+i}-gC_{F-g}.
\end{equation}%
The term $gC_{F-g}$ outside the summation vanishes by symmetry: since $g$ is
a gap, $F-g$ must be in $S$ and therefore $C_{F-g}=0$. The terms $%
p_{i}C_{F-g+i}$ inside the summation each vanish either by induction or by
symmetry combined with closure. Indeed, either $i$ is in $G$ and therefore $%
p_{i}=0$ by induction, or $i$ is in $S$ along with $F-g$ and $F-g+i=\left(
F-g\right) +i$, implying that $C_{F-g+i}=0$. This completes the proof for a
symmetric $S$.

For a pseudo-symmetric $S$, the last part of the above argument fails for $%
g=F/2$ since, for this particular value, $F-g$ is not in $S$. However, the
argument is rescued by the fact mentioned in the Introduction that $i$ being
in $S$ implies that $F+i$ is also in $S$. With this in mind, let us
re-examine the equation for $g=F/2$:

\begin{equation}
p_{F/2}=-\sum_{i=1}^{F/2-1}p_{i}C_{F/2+i}-\frac{F}{2}C_{F/2}.
\end{equation}%
The term outside the summation vanishes since we specifically excluded the $%
F/2$ term from pseudo-gap polynomial. On the other hand, each of the terms $%
p_{i}C_{F/2+i}$ vanishes because either $i$ is in $G$ and therefore $p_{i}=0$
by induction, or $i$ is in $S$ along with $F/2+i$ and therefore $C_{F/2}+i=0$%
. This completes the proof for a pseudo-symmetric $S$.

\begin{remark}
If the $z^{F/2}$ term were allowed in a pseudo-gap polynomial, all the gap
power sums $p_{g}$ except $p_{F/2}$ and $p_{F}$ would still vanish.
\end{remark}

\section{Corollaries}

A number of corollaries readily follow from the above theorem. Let $r$
denote the set of the $F-m$ roots of $P^{\left( m\right) }\left( z\right) $.

\begin{corollary}
Any power sum product $p_{\lambda }=p_{\lambda _{1}}\cdots p_{\lambda _{l}}$
of $F-m$ variables vanishes at $r$ if at least one of $%
\lambda _{1},\cdots ,\lambda _{l}$ is a gap of $S$. The proof of this
corollary is immediate.
\end{corollary}

\begin{corollary}
Any homogeneous symmetric polynomial $Q\left( z_{1},\cdots ,z_{F-m}\right) $
of degree $g\in G$ vanishes at $r$. In particular, any
elementary symmetric polynomial $e_{g}\left( z_{1},\cdots ,z_{F-m}\right) $
vanishes at $r$.

\begin{proof}
By the fundamental theorem of symmetric polynomials, $Q$ can be written as a
linear combination of power sum products corresponding to the integer
partitions of $g$, i.e.
\begin{equation}
Q\left(  z_{1},\cdots,z_{F-m}\right)  =\sum_{\lambda}c_{\lambda}p_{\lambda
}\left(  z_{1},\cdots,z_{F-m}\right)
\end{equation}
where $p_{\lambda}=p_{\lambda_{1}}\cdots p_{\lambda_{n}}$ and $\lambda
_{1}+\cdots+\lambda_{n}=g$. At least one of $\lambda_{1},\cdots,\lambda_{n}$
must be a gap since otherwise their sum would lie in $S$ by additive closure.
An appeal to the previous corollary completes the proof.
\end{proof}
\end{corollary}

\begin{corollary}
The sum of the critical values of $P^{\left( m-1\right) }\left( z\right) $
vanishes, i.e. 
\begin{equation}
\sum_{i=1}^{F-m}P\left( r_{i}\right) =0.
\end{equation}
\end{corollary}

\begin{proof}
We have,
\begin{equation}
\sum_{i=1}^{F-m}P\left(  r_{i}\right)  =\sum_{i=1}^{F-m}\sum_{g\in G}
c_{g}r_{i}^{g}.
\end{equation}
Interchanging the order of summation yields
\begin{equation}
\sum_{i=1}^{F-m}P\left(  r_{i}\right)  =\sum_{g\in G}c_{g}\sum_{i=1}
^{F-m}r_{i}^{g}=\sum_{g\in G}c_{g}p_{g}=0,
\end{equation}
as we set out to show.
\end{proof}

\section{Conclusion}

We have described a novel connection, captured by equation (\ref{pg})
between symmetric and pseudo-symmetric numerical semigroups and the theory
of symmetric polynomials. We expect this connection will enrich both fields
and open new avenues for future investigations.

\bibliographystyle{abbrv}
\bibliography{Classics,NSG}

@book{Assi2016,
  author    = {Assi, Abdallah and Garc{\'i}a-S{\'a}nchez, Pedro A.},
  title     = {Numerical Semigroups and Applications},
  series    = {RSME Springer Series},
  volume    = {1},
  publisher = {Springer},
  year      = {2016},
  doi       = {10.1007/978-3-319-41333-4}
}

@article{BorziHerreraPoyatosMoree2021,
  author  = {Borz{\`i}, A. and Herrera-Poyatos, A. and Moree, P.},
  title   = {Cyclotomic numerical semigroup polynomials with at most two irreducible factors},
  journal = {Semigroup Forum},
  volume  = {103},
  pages   = {812--828},
  year    = {2021},
  doi     = {10.1007/s00233-021-10197-8}
}

@article{Ciolan2016Cyclotomic,
  author  = {Ciolan, Emil-Alexandru and Garc{\'i}a-S{\'a}nchez, Pedro A. and Moree, Pieter},
  title   = {Cyclotomic Numerical Semigroups},
  journal = {SIAM Journal on Discrete Mathematics},
  volume  = {30},
  number  = {2},
  pages   = {650--668},
  year    = {2016},
  doi     = {10.1137/15M1010027}
}

@article{Fel2015,
  author    = {Fel, Leonid G.},
  title     = {Genera of Numerical Semigroups and Polynomial Identities},
  journal   = {Symmetry, Integrability and Geometry: Methods and Applications (SIGMA)},
  volume    = {11},
  pages     = {065},
  year      = {2015},
  doi       = {10.3842/SIGMA.2015.065}
}

@article{Fel2021,
  author    = {Fel, Leonid G.},
  title     = {Symmetric Polynomials Associated with Numerical Semigroups},
  journal   = {Discrete Mathematics Letters},
  volume    = {5},
  pages     = {56--62},
  year      = {2021},
  doi       = {10.47443/dml.2020.0053}
}

@article{Moree2014Numerical,
  author  = {Moree, Pieter},
  title   = {Numerical Semigroups, Cyclotomic Polynomials, and {B}ernoulli Numbers},
  journal = {The American Mathematical Monthly},
  volume  = {121},
  number  = {10},
  pages   = {890--902},
  year    = {2014},
  doi     = {10.4169/amer.math.monthly.121.10.890}
}

@book{RamirezAlfonsin2005,
  author    = {Ram{\'i}rez Alfons{\'i}n, Jorge Luis},
  title     = {The Diophantine Frobenius Problem},
  series    = {Oxford Lecture Series in Mathematics and Its Applications},
  volume    = {30},
  publisher = {Oxford University Press},
  year      = {2005},
  doi       = {10.1093/acprof:oso/9780198568209.001.0001}
}

@book{Rosales2009,
  author    = {Rosales, Jos{\'e} Carlos and Garc{\'i}a-S{\'a}nchez, Pedro A.},
  title     = {Numerical Semigroups},
  series    = {Developments in Mathematics},
  volume    = {20},
  publisher = {Springer},
  year      = {2009},
  doi       = {10.1007/978-1-4419-0160-6}
}

\end{document}